\documentclass[letter,11pt,oneside,reqno]{amsart} 
\usepackage[margin=1in]{geometry}
\usepackage{amsfonts,amssymb,amsmath,amsthm}
\usepackage{graphicx,color}
\usepackage[utf8]{inputenc}
\usepackage{hyperref, dsfont}
\usepackage{pdfsync, enumerate}
\allowdisplaybreaks
\definecolor{seagreen}{rgb}{0.09, 0.45, 0.27}
\definecolor{purple}{rgb}{0.63, 0.36, 0.94}

\newtheorem{theorem}{Theorem}[section]

\newtheorem{proposition}[theorem]{Proposition}
\newtheorem{corollary}[theorem]{Corollary}

\theoremstyle{definition}
\newtheorem{definition}[theorem]{Definition}

\newtheorem{example}{Example}
\newtheorem{remark}[theorem]{Remark}

\numberwithin{equation}{section}

\def\<{\langle}
\def\>{\rangle}
\newcommand{\equalD}{\overset{\mathcal D}{=\joinrel=}}

\newcommand{\E}{\ensuremath{\mathbb{E}}}

\newcommand{\zz}[1]{\mathbb{#1}}

\newcommand{\Prob}{\mathbb P}

\makeatletter
\newcommand*{\mathdef}{{\mathrel{\rlap{%
                     \raisebox{0.3ex}{$\m@th\cdot$}}%
                     \raisebox{-0.3ex}{$\m@th\cdot$}}%
                     =\,}}
\makeatother

\DeclareUnicodeCharacter{2147}{\dash}

\def\1{\ifmmode {1\hskip -3pt \rm{I}}
\else {\hbox {$1\hskip -3pt \rm{I}$}}\fi} 

\newcommand{\fise}{f_{\textup{ISE}}}

\newcommand{\xset}[1]{
  \left\{ #1 \right\}}


\title[Occupation densities of BRWs]{Occupation densities of Ensembles of Branching Random Walks}
\author[S. Lalley and S. Tang]{Steven P. Lalley and Si Tang}
\thanks{The first author is supported by National Science Foundation Award DMS-1612979. }
\address{Steven P. Lalley\\
  Department of Statistics\\
  University of Chicago\\
  5747 S. Ellis Avenue\\
  Chicago, IL 60637, USA}
\email{lalley@galton.uchicago.edu}

\address{Si Tang\\
  Department of Mathematics\\
  Lehigh University\\
  17 Memorial Drive East\\
  Bethlehem, PA 18015, USA}
\email{sit218@lehigh.edu}

\begin{document}
\date{\today}

\maketitle

\begin{abstract} We study the limiting occupation density process for
  a large number of critical and driftless branching random walks. We
  show that the rescaled occupation densities of $\lfloor sN\rfloor$
  branching random walks, viewed as a function-valued, increasing
  process $\{g_{s}^{N}\}_{s\ge 0}$, converges weakly to a pure jump
  process in the Skorohod space
  $\mathbb D([0, +\infty), \mathcal C_{0}(\zz R))$, as
  $N\to\infty$. Moreover, the jumps of the limiting process consist of
  i.i.d. copies of an Integrated super-Brownian Excursion (ISE)
  density, rescaled and weighted by the jump sizes in a real-valued
  stable-1/2 subordinator.
\end{abstract}

\section{Introduction}
\label{sec:intro}
In a branching random walk on the integers, individuals live for one generation,
reproduce as in a Galton-Watson process, giving rise to offspring
which then independently jump according to the law of a random walk.
A branching random walk is said to be \emph{critical} if the offspring
distribution $\nu$ has mean $1$, and \emph{driftless} if the jump
distribution $F$ has mean $0$ and finite variance. We will assume
throughout that (i) the offspring distribution $\nu$ has mean one (so
that the Galton-Watson process is critical) and finite, positive
variance $\sigma^{2}_{\nu} $; and (ii) the step distribution $F$ for
the random walk has span one, mean zero and finite, positive variance
$\sigma^{2}_{F}$. (Thus, the spatial locations of individuals will always
be points of the integers $\mathbb{Z}$.)

To any  branching random walk can be associated a randomly labeled
Galton-Watson tree $\mathcal{T}$, where the Galton-Watson tree describes the lineage
of the individuals and the label of each vertex marks the spatial location of
the corresponding individual. This  labeled tree
\[\mathcal T= (T, \{l(v)\}_{ v\in T})\] 
is generated as follows. 
\begin{enumerate}[(i)]
\item Let $T$ be the genealogical tree of a Galton-Watson process with
  a single ancestral individual and offspring distribution $\nu$, with
  the root node $\rho$ representing this ancestral individual. Since
  $\nu$ has mean $1$, the tree $T$ is finite with probability one.
\item Assign the label $l(\rho)=0$ to the root.
\item
  Conditional on $T$, let $\xset {\xi_{e}}_{e \in \mathcal{E}(T)}$ be
  a collection of i.i.d. random variables, with common distribution
  $F$ (the ``step distribution'') indexed by the (directed)
  edges $e=(u,v)$ of the tree $T$, where $u$ is the parent vertex of $v$. 
  For any such directed edge
  $e=(u,v)$, define
  \begin{displaymath}
    l(v)=l(u)+\xi_{e}.
  \end{displaymath}
\end{enumerate}

Given the labeled tree $\mathcal T$ associated with the branching
random walk, the occupation measures can be recovered as follows. For
any time $n \in \mathbb{Z}_{+}$ and any site $x\in \mathbb{Z}$, the
number $Z_{n}(x)$ of individuals at location $x$ at time $n$ is the
number of vertices $v \in T$ at height $n$ (i.e., at distance $n$ from
the root) with label $l(v)=x$. The \textit{vertical
  profile}, or the \textit{occupation measure}, of $\mathcal T$ (see
\cite{ISE2006-2}) is the random counting measure on $\mathbb Z$
defined by
\begin{displaymath}
  X(x; \mathcal{T})=\sum _{n=0}^{\infty}Z_{n}(x).
\end{displaymath}
In this paper, we study the limiting behavior of
the occupation measure and its connection to super-Brownian
motions. Before stating our main result, we review a few results about
the occupation measure of random labeled trees.

The study of such occupation measures dates back to Aldous
\cite{AldousISE1993}, who introduced an object called the \emph{integrated
super-Brownian excursion} (ISE), denoted by  $\mu_{\textup{ISE}}$, a (probability)
measure-valued random variable that arises as the scaling limit
of the occupation measure of certain labeled random planar trees and
tree embeddings. In \cite{marckert2004rotation}, Marckert proved that
the rescaled occupation measure of a random binary tree of $n$
vertices converges weakly to $\mu_{\text{ISE}}$ as $n\to
\infty$. 

Bousquet-M\'{e}lou and Janson \cite{ISE2006-2} later proved a local
version of Marckert's result: they showed that the \emph{density} of
the rescaled occupation measure of random binary trees, random
complete binary trees, or random plane trees on $n$ vertices converges
to the density of $\mu_{\textup{ISE}}$, denoted by $\fise$, which is known to be a 
random H\"older$(\alpha)$-continuous function for every $\alpha <1$ 
with compact support. This local convergence was later extended in \cite[Theorem 1.1]{condGW-profile}
to general branching random walks conditioned to have exactly $n$
vertices, as long as $\nu$ and $F$ satisfy the assumptions above.

In this paper, we consider an ensemble of critical, driftless
branching random walks, all with the same offspring and step
distributions $\nu$ and $F$, and study the limiting behavior of the total
occupation density. Our first result shows that the total occupation
density converges in the Skorohod space
$\mathbb D([0, +\infty), \mathcal C_{0}(\mathbb R)\,)$, which can be
characterized by a super-Brownian
motion.

Let $\mathcal T^{1}, \mathcal{T}^{2},\ldots$ be the random labeled
trees associated with an infinite sequence $(Z_{1,n})_{n\ge 0}$, $(Z_{2,n})_{n\ge 0}, \cdots $ of
independent copies of the branching random walk. For each integer
$m \geq 1$ define
\begin{equation}
  \label{eq:occupation-measure}
  X^{m}(j):= \sum_{i=1}^{m} X(j;\mathcal{T}^{i})
\end{equation}
to be the total number of vertices in the first $m$ trees with label
$j\in \mathbb Z$, and define $\bar X^{m}(x) $ to be the linear
interpolation to $x \in \mathbb{R}$. Observe that $X^{m}$ can be
viewed as the occupation measure of the branching random walk
initiated by the $m$ ancestral particles that engender the branching
random walks $Z_{1},Z_{2}, \cdots Z_{m}$. Clearly, the function
$\bar X^{m}(x) $ is an element of $ \mathcal {C}_{0}(\mathbb
R)$. Finally, define the $\mathcal C_{0}(\mathbb R)$-valued process
$\{g_{s}^{N}\}_{s\ge 0}$ by
\begin{equation}
  \label{eq:occupation-measure-rescaled}
  g_{s}^{N}(x) :=N^{-3/2} \bar X^{\lfloor sN\rfloor}(\sqrt{N} x).
\end{equation}

\begin{theorem} \label{thm1.1} As $N \to \infty$, the rescaled density
  processes $\{ g_{s}^{N}\}_{s\ge 0} $ converge weakly in the Skorohod
  space
  $\mathbb D:=\mathbb D([0, +\infty), \mathcal C_{0}(\mathbb R))$ to a
  process $\{g_{s}\}_{s\ge 0}$.  Moreover, the limiting process
  satisfies
\begin{align}
\label{eqn:defgs}
g_{s}(x) &\equalD \int_{0}^{\infty}Y^{s}\left(t,
x\right) dt
\end{align}
where $\{Y^{s}(t, x), x\in \zz R\}_{t\ge
  0}$ is the density process for a super-Brownian motion $\{Y^{s}_{t}\}_{t\ge 0}$ with variance
parameters $(\sigma_{\nu}^{2}, \sigma_{F}^{2})$, started from the initial measure
$Y^{s}_{0}=s \delta_{0}$.
\end{theorem}

\begin{remark} Super-Brownian motion $\{Y_{t}\}_{t\ge 0}$ is, by
  definition (see for instance \cite{etheridge2000introduction}, ch.~1) a
  measure-valued stochastic process that can be constructed as a weak
  limit of rescaled counting measures associated with branching random
  walks.  In one dimension, for each $t>0$, the random measure $Y_{t}$
  is absolutely continuous relative to the Lebesgue measure, and the
  Radon-Nikodym derivative $Y(t, x)$ is jointly continuous in $(t,x)$
  \cite{KonnoShiga1988}. Super-Brownian motion is singular in 
  higher dimensions and thus the representation \eqref{eqn:defgs} does 
  not exist in higher dimensions. 
  When the dependence on the variance
  parameters $\sigma_{\nu}^{2}$ and $\sigma_{F}^{2}$ must be
  emphasized, we do so by adding them as extra superscripts, i.e.,
  \begin{displaymath}
    Y^{s,\sigma_{\nu}^{2},\sigma_{F}^{2}}(t,x).
  \end{displaymath} When
  $\sigma_{\nu}^{2}=\sigma^{2}_{F}=1$, the measure-valued process
  associated with $Y^{s,1,1}(t,x)$ is a
  \emph{standard} super-Brownian motion. The density processes for
  different variance parameters obey a simple scaling
  relation:
\[
 Y^{s, \sigma_{\nu}^{2}, \sigma_{F}^{2}}\left(t, x\right) \equalD
 \sigma_{\nu}\sigma_{F}^{-1}Y^{s,1,1}\left(\sigma_{\nu}^{2}t,
   \sigma_{\nu}\sigma_{F}^{-1}x\right),\ \  \quad \textrm{for all} \, t>0,
 \, x \in \mathbb{R}.
\]
Thus,  we can rewrite \eqref{eqn:defgs} in terms of the density
function of standard super-Brownian motion as follows:
\[
g_{s}(x) \equalD \sigma_{\nu}\sigma_{F}^{-1}
\int_{0}^{\infty}Y^{s,1,1}\left(\sigma_{\nu}^{2}t,
  \sigma_{\nu}\sigma_{F}^{-1}x\right)dt 
\]
\end{remark}

\begin{remark}\label{rem:indef-integral}
  For any fixed $s>0$ and each integer $N\geq 1$, the random function
  $g^{N}_{s} (\cdot)$ is the (rescaled) occupation density of the
  branching random walk gotten by amalgamating the branching random
  walks generated by the first $\lfloor sN\rfloor$ initial
  particles. Because this sequence of branching random walks is
  governed by the fundamental convergence theorem of Watanabe
  \cite{Watanabe1968} and its extension to densities by Lalley
  \cite{lalley1d}, the limiting random function $g_{s} (\cdot)$ must
  (after the appropriate scaling) be the integrated occupation density of the
  super-Brownian motion with initial measure $s\delta_{0}$. This
  explains relation \eqref{eqn:defgs}. But even for fixed $s>0$ the
  weak convergence
  $g^{N}_{s}(x) \Longrightarrow\int_{0}^{\infty}Y^{s,
    \sigma_{\nu}^{2}, \sigma_{F}^{2}}\left(t, x\right) dt$ does
  not follow directly from the local convergence of the density
  process proved in \cite[Theorem 2]{lalley1d}, for two
  reasons. First, the local convergence result in \cite{lalley1d}
  requires that the initial densities must, after Feller-Watanabe
  rescaling, converge to a density function
  $Y^{s}(0, \cdot)\in \mathcal C_{0}(\mathbb R)$. In Theorem
  \ref{thm1.1}, however, the limiting initial density
  $Y^{s}_{0}=s \delta_{0}$ is not absolutely continuous with respect
  to the Lebesgue measure. Second, even if the local convergence could
  be shown to remain valid under the initial condition
  $Y^{s}_{0}=s\delta_{0}$, the indefinite integral operator on
  $\mathbb C([0, +\infty), \mathcal C_{0}(\mathbb R))$ is not bounded,
  and so it would not follow, at least without further argument, that
  the integral of the discrete densities would converge to that of the
  super-Brownian motion density over the time interval $[0, +\infty)$.
\end{remark}

\begin{figure}[htbp]
\centering
\includegraphics[width=0.8\textwidth]{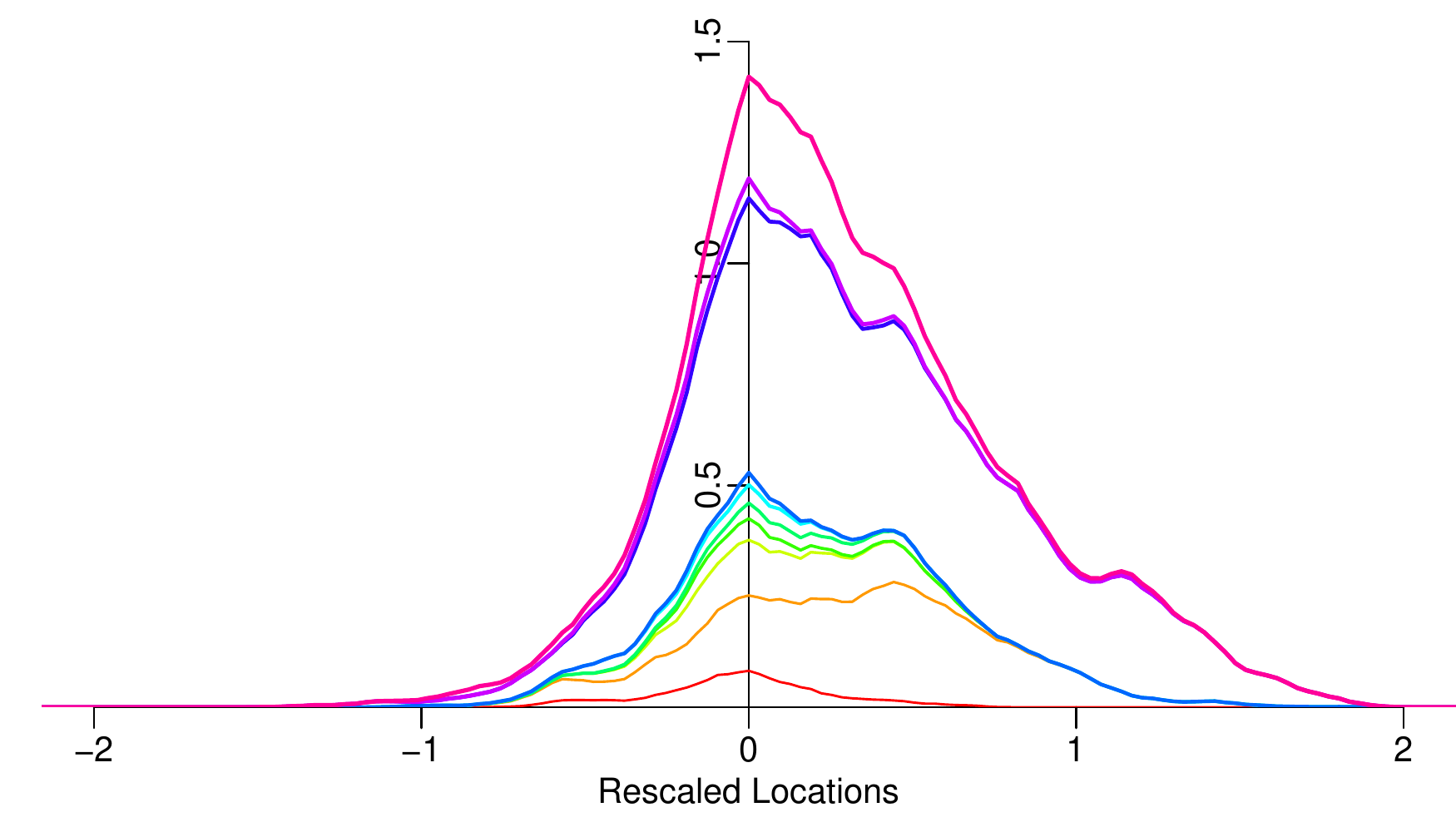}
\caption{A simulation of $g_{s}^{N}$, for $N=1000$ and $s=0.1, 0.2, \ldots, 1$. The offspring and step distributions are $\nu = \text{Poi}(1)$ and $F = (\delta_{-1}+\delta_{0}+\delta_{1})/3$.}
\label{Fig:simuG}
\end{figure}

For each $N\geq 1$, the process $\{g_{s}^{N}\}_{s\ge 0}$ is
nondecreasing\footnote{By ``nondecreasing'', we mean that for all
  $t\ge s\ge 0$, $g^{N}_{t}-g^{N}_{s} \ge 0$. A simulation is shown in
  Figure \ref{Fig:simuG}. } in $s$ (relative to the natural partial
ordering on $\mathcal C_{0} (\zz{R})$) and has stationary, independent
increments.  Therefore, the limiting process $\{g_{s}\}_{s\ge0}$ must
also be nondecreasing, with stationary, independent increments. 
We prove the following properties of the limiting process.

\begin{theorem}\label{thm:property-gs}
The limiting function-valued process $\{g_{s}\}_{s\ge 0}$ has the following properties.
\begin{enumerate}[\textup{(}i\textup{)}]
\item It obeys the scaling relation $g_{s}(x)
  \overset{\mathcal D}{=\joinrel=} s^{3/2}g_{1}(x/\sqrt{s})$. 
\item  The real-valued process $\{I_{s}\}_{s\ge 0}$, where  $I_{s} := g_{s}(0)$ is
  the occupation density at zero, is a stable subordinator with exponent $\alpha = 2/3$.
\item The real-valued process $\{\theta_{s}\}_{s\ge 0}$, where
  $\theta_{s} := \int_{-\infty}^{\infty}g_{s}(x)dx$ is the
  total rescaled occupation density, is a stable subordinator with
  exponent $\alpha = 1/2$.
\item $\{g_{s}\}_{s\ge 0}$ is a pure-jump subordinator in the Banach lattice $\mathcal C_{0}(\mathbb
R)$ (see Definition \ref{def:banachlattice}). 
\end{enumerate} 
\end{theorem}

As we will show, the limiting process $g_{s}$ in Theorem
\ref{thm1.1} has jump discontinuities, that is, there are times
$t>0$ such that the function $g_{t}-g_{t-}$ is non-zero (and hence
positive) over some interval. The jumps that occur before time $t=1$ can
be ordered by total area $\int (g_{t}-g_{t-})
(x)\,dx$, i.e., the jump size in the stable-1/2 process $\{\theta_{s}\}_{s\ge 0}$. Denote these jumps (viewed as elements of $\mathcal C_{0} (\zz{R})$)
by
\[
	J_{1} (x), J_{2} (x),\dots \quad \text{where} \quad \int
	J_{1}>\int J_{2}>\int J_{3}>\dots . 
\]
(In Section \ref{sec:Levyjump}, we will see that no two  jump sizes can be the same.)
For each $N$, the Galton-Watson trees $\mathcal{T}_{i}$ with
$i\leq N$ can also be ordered by their size (i.e., the number of vertices). The corresponding jumps in the (rescaled) occupation
density $g^{N}_{s}$ will be denoted by
\[
	J^{N}_{1} (x), J^{N}_{2} (x), J^{N}_{3} (x),\dots .
\]
(Thus, if the $j$th largest tree among the first $N$
trees is $\mathcal T_{\lfloor s_{j}N\rfloor} $, then $J^{N}_{j}=
g^{N}_{s_{j}}-g^{N}_{\frac{\lfloor s_{j}N\rfloor-1}{N}}$.)
\begin{corollary}
\label{cor:order-stats}
For each $m\geq 1$,
\begin{equation}\label{eq:jointConvergence}
	(J^{N}_{1},J^{N}_{2},\dots ,J^{N}_{m})\Longrightarrow 
	(J_{1},J_{2},\dots ,J_{m}),
\end{equation}
where the weak convergence is relative to the $m$-fold product
topology on $\mathcal C_{0} (\zz{R})$.
\end{corollary}

\begin{proof}
This is an immediate consequence of Theorem \ref{thm1.1},
because weak convergence in the Skorohod topology on $\zz{D}$ implies
weak convergence of the ordered jump discontinuities.
\end{proof}

Theorem \ref{thm1.1} can be regarded as an unconditional 
version of the local convergence in \cite[Theorem 3.1]{ISE2006-2} and \cite[Theorem 1.1]{condGW-profile}. The connection between Theorem \ref{thm1.1} and the results of
Bousquet-M\'elou/Janson and Aldous leads to a reasonably complete
description of the L\'{e}vy-Khintchine representation of the
pure-jump process $\{g_{s}\}_{s\ge 0}$. 

\begin{theorem} 
\label{thm:levy-measure} The point process of jumps of $\{g_{s}\}_{s\ge 0}$ is a Poisson point process
$\{\mathcal N(B)\}_{B \in \mathcal B}$  on the space $(\mathbb R_{+}\times
\mathbb R_{+}\times \mathcal C_{0}(\mathbb R), \mathcal B:= \mathcal B_{\mathbb
R_{+}\times \mathbb R_{+}\times \mathcal C_{0}(\mathbb R)})$ with
intensity (L\'{e}vy-Khintchine) measure
\begin{equation}
\label{eqn:gs-intensity}
\chi(dt, dl, dh) := dt\cdot \frac{dl}{\sqrt{2\pi l^{3}}}\cdot f_{\textup{ISE}}(dh). 
\end{equation}
Consequently, the process $g_{s}$ can be written as
\begin{equation}
\label{eqn:pointprocess}
g_{s}(\cdot) \equalD \frac{1}{\sigma_{\nu}\sigma_{F}}\iiint \mathbf 1_{[0,
s]}(t) l^{3/4}h\left(l^{-1/4}\,\frac{ \sigma_{\nu}}{\sigma_{F}} \cdot\,
\right) \mathcal N(dt, dl, dh). 
\end{equation}
\end{theorem}

The remainder of the paper is organized as follows. Section \ref{sec:proof1.1} is devoted
to the proof of Theorem \ref{thm1.1}, where we make use of Aldous'
stopping time criterion \cite{Aldous-tightness} to show the tightness
of the sequence of processes $\{g^{N}_{s}\}_{s\ge 0}$. In Section \ref{sec:Levyjump},
we prove the properties of the limiting process $g_{s}$ enumerated in
Theorem \ref{thm:property-gs} and the
L\'{e}vy-Khintchine representation \eqref{eqn:pointprocess}.

\section{Proof of Theorem \ref{thm1.1}\label{sec:proof1.1}}
\subsection{Preliminaries on the Skorohod space $\mathbb D$}
Let $(\mathcal S, d)$ be a separable and complete metric
space, and  let
$\mathbb D (\mathcal S)\mathdef \mathbb D ([0, +\infty), (\mathcal S,
d))$ be the spaces of all $\mathcal S$-valued \emph{c\`{a}dl\`{a}g}
functions $f$ with domain $[0, +\infty)$, i.e., $f\in \mathbb D(\mathcal S)$ is
right-continuous and has left
limits. 
The space $\mathbb D(\mathcal S)$ is metrizable, and under the usual
\emph{Skorohod metric}, the space $\mathbb D(\mathcal S)$ is complete
and separable. We refer to \cite[Section 13]{Billingsley1999} for
details on the Skorohod topology. Here, we quote the following
theorem, which gives a sufficient condition for the weak convergence
in $\mathbb D(\mathcal S)$.

\begin{theorem}\label{thm:convD} Let $\{X_{t}^{N}\}_{t \ge 0}$ and
  $\{X_{t}\}_{t \ge 0} \in \mathbb D(\mathcal S)$ be
  $\mathcal S$-valued processes. Let $\Delta$ be some dense subset of
  $[0, +\infty)$. If the sequence $\{X^{N}\}_{t\ge 0}$ is tight
  (relative to the Skorohod metric) and if
  $(X^{N}_{t_{1}}, X^{N}_{{t_{2}}}, \ldots, X_{t_{m}}^{N})
  \Longrightarrow (X_{t_{1}}, X_{{t_{2}}}, \ldots, X_{t_{m}})$ as
  $N \to \infty$ for all $t_{1}, \ldots, t_{m} \in \Delta$, then
  $\{X^{N}_{t}\}_{t\ge 0 }\Longrightarrow \{X_{t}\}_{t \ge 0}$ as
  $N \to \infty$.
\end{theorem}
For the particular case 
where $(\mathcal S, d)$ is the real line with the Euclidean metric, 
 Aldous \cite[Theorem 1]{Aldous-tightness}  gave a useful sufficient condition
for the tightness of a sequence
$\{X^{N}_{t}\}_{t \in [0,1]}$ in the space
$\mathbb D([0, 1], (\mathbb R, \,\lvert\cdot\rvert))$.  He also
pointed out  \cite[Theorem 4.4]{aldous1981weak} that, with a slight
modification, the criterion can be generalized to $\mathcal S$-valued
stochastic processes over the half line $[0, +\infty)$ as long as
$(\mathcal S, d)$ is a complete and separable metric space. We state
Aldous' criterion in this form below.
\begin{theorem}\label{thm:tightness}
  Let $(\mathcal S, d)$ be a complete and separable metric space and
  let $\{X^{N}_{s}\}_{s\ge 0} \in \mathbb D ([0, +\infty), (\mathcal
  S, d))$ be a sequence of $\mathcal S$-valued stochastic
  processes. A sufficient condition for tightness of the sequence
  $\xset {X^{N}_{s}}_{s \geq 0}$ is that the following two conditions
  hold:

  \medskip \noindent 
  \underline{\text{Condition }$ 1^{\circ}$}. For each $s$, the
  sequence 
  $\xset {X_{s}^{N}}_{N\in \mathbb N}$ is tight in $(\mathcal S, d)$, and\\
  \underline{\text{Condition }$ 2^{\circ}$}. For any $L > 0$, any
  sequence of constant $\delta_{N} \downarrow 0$, and any sequence of
  stopping times $\tau_{N} $ for
  $\{X_{s}^{N}\}_{s\ge 0}$ that are all upper bounded by $L$,
\begin{equation}
\label{eq:condition2}
d(X_{\tau_{N}+\delta_{N}}^{N}, X_{\tau_{N}}^{N}) \stackrel{P}{\longrightarrow} 0.
\end{equation}
\end{theorem}
\begin{remark} In the case when $(\mathcal S, d) = (\mathcal C_{0}(\zz
  R), \lVert\cdot\lVert_{\infty})$, \eqref{eq:condition2} is
  equivalent to  
\begin{equation}
\label{eq:condition2-1}
\lVert X_{\tau_{N}+\delta_{N}}^{N} - X_{\tau_{N}}^{N} \rVert_{\infty}\stackrel{P}{\longrightarrow} 0.
\end{equation}
\end{remark}

\subsection{Proof of Tightness} In this section, we prove that the
sequence $\{g^{N}_{s}\}_{N\in \mathbb N}$ is tight in
$\mathbb D (\mathcal C_{0}(\mathbb R))$ by verifying Condition
1$^{\circ}$ and Condition $2^{\circ}$.

\bigskip 
To verify { Condition $1^{\circ}$,  we will show that for any $s\ge 0$ fixed
and any $\epsilon > 0$, we can find a compact subset
$\mathsf K \subset \mathcal C_{0}(\mathbb R)$ such that
$\Prob (g_{s}^{N}\in \mathsf K) > 1-\epsilon$ for all $N$ large. Let
$\zeta_{sN}$ be the extinction time of the branching random walk
gotten by amalgamating the branching random walks $(Z_{1,n})_{n\ge 0}, (Z_{2,n})_{n\ge 0}, \cdots
,(Z_{\lfloor sN\rfloor, n})_{n\ge 0}$, that is, $\zeta_{sN}$ is the maximum of the
extinction times of the branching random walks
initiated by the first $\lfloor sN\rfloor$ ancestral particles. By a
fundamental theorem of Kolmogorov,
\begin{displaymath}
  P \xset {\zeta_{1}>n} \sim \frac{2}{n \sigma_{\nu}^{2}} \quad
  \textrm {as} \; n \rightarrow\infty;
\end{displaymath}
consequently, for 
every $\epsilon > 0$, there exists  $H=H_{s, \epsilon}>0$ such that
\begin{equation}
  \label{eq:kolmogorov-tail}
  \Prob(\zeta_{sN} > NH) < \epsilon /2.
\end{equation}
Therefore, it suffices
to prove that there is a compact set  $\mathsf K \subset \mathcal
C_{0}(\mathbb R)$ such that for all $N$ large,
\begin{equation}
  \label{eq:tightness1}
  \Prob\left(\{g_{s}^{N}\in \mathsf K\} \cap G_{sN}\right) \ge
  1-\frac{\epsilon}{2},  \quad \textrm {where} \quad
  G_{sN}:= \xset {\zeta_{sN}\le NH}.
\end{equation}

To establish inequality \eqref{eq:tightness1} we will use
Kolmogorov-\v{C}entsov criterion (see, e.g., \cite[Chapter 2, Problem
4.11]{GTM-BM1991}). It suffices to prove
that 
\begin{align}
\label{eqn:momentbd0}
\tag{2.5a}
\sup_{N}\E[g_{s}^{N}(0)\mathbf 1_{G_{sN}}] < \infty,
\end{align} and that for some $m\ge 3$, there exists $C = C(s, m, H)>0$ such that for
all $x, y \in\mathbb R$ and for all $N$ sufficiently large,
\begin{align}
\label{eqn:momentbd}
\tag{2.5b}
\mathbb E \left [g_{s}^{N}(x) - g_{s}^{N}(y)\right ]^{2m}
  \mathbf 1_{G_{sN}}\le C |x-y|^{\frac{2m}{5}}.  
\end{align}
Note that the requirement $m \geq 3$ in \eqref{eqn:momentbd} ensures that the exponent $2m/5$ is
larger than $1$, as is needed for the Kolmogorov-\v{C}entsov criterion. We will rely on the following estimates of
\cite{lalley1d} to compute these bounds.

\begin{proposition}\label{prop:lalley71} \cite[Proposition 5]{lalley1d}
  Let $Z_{n}(x)$ be the number of particles at location
  $x\in \mathbb Z$ and time $n\in \mathbb Z_{+}$ in a branching random
  walk, started from a single particle at $0\in\mathbb Z$, with
  offspring distribution $\nu$ and step distribution $F$. For each
  $m\in \mathbb N$, there is constant $C_{m}$ such that for all
  $x, y\in \mathbb Z$ and all $ n\ge 1$,
\begin{align}
\setcounter{equation}{5}
\E Z_{n}(x)^{m}&\le C_{m}n^{m/2-1} \label{eqn:lalley70} \\
|\E(Z_{n}(x) - Z_{n}(y))^{m}| &\le C_{m}n^{2m/5-1}|x-y|^{m/5} 
\label{eqn:lalley71}
\end{align}
\end{proposition}

\begin{corollary}
  \label{prop:claim} Let $Z_{n}(x)$ be  as in Proposition
  \ref{prop:lalley71}. Then, for any $H > 0$ and  $m \ge 1$, there exists
  $C=C_{m, H}$ such that for all $N\ge 1$ and
  $x, y \in \mathbb Z/\sqrt{N}$,
\begin{align}
\label{eqn:claim}
\sum_{n_{1}, \ldots, n_{m}=1}^{\lfloor NH \rfloor} \left |  \mathbb E
  \prod_{l=1}^{m} \left( Z_{n_{l}} (\sqrt{N}x)-Z_{n_{l}}
  (\sqrt{N}y)\right) \right | \le C |x-y|^{\frac{m}{5}}N^{\frac{3m}{2}-1}. 
\end{align} 
\end{corollary}
\begin{proof}
Suppose first $m$ is even. Then by a trivial extension of  H\"older's
inequality (see, e.g.,\cite[pp. 4]{Pons-Ineq2013}) and Proposition~\ref{prop:lalley71},
\begin{align*}
& \ \ \ \ \ \sum_{n_{1}, \ldots, n_{m}=1}^{\lfloor NH \rfloor}  \left
                 | \mathbb E \prod_{l=1}^{m} \left( Z_{n_{l}}
                 (\sqrt{N}x)-Z_{n_{l}} (\sqrt{N}y)\right)  \right | 
  \\
& \le \sum_{n_{1}, \ldots, n_{m}=1}^{\lfloor NH \rfloor} \prod_{l=1}^{m} \left | \mathbb E\left( Z_{n_{l}} (\sqrt{N}x)-Z_{n_{l}} (\sqrt{N}y) \right)^{m} \right|^{\frac{1}{m}}\\
& =  \left \{ \sum_{n_{1}=1}^{\lfloor NH \rfloor}  \left |  \mathbb E\left( Z_{n_{1}} (\sqrt{N}x)-Z_{n_{1}} (\sqrt{N}y) \right)^{m} \right|^{\frac{1}{m}} \right \}^{m}\\
&\le \left \{ \sum_{n_{1}=1}^{\lfloor NH \rfloor}  \left [ C_{m}n_{1}^{\frac{2m}{5}-1}|x-y|^{\frac{m}{5}}N^{\frac{m}{10}}\right]^{\frac{1}{m}} \right \}^{m}\\
&\le C_{m, H}|x-y|^{\frac{m}{5}}N^{\frac{3m}{2}-1}.
\end{align*}
The case when $m$ is odd is similar.
\end{proof}

\begin{proof}[Proof of \textup{Condition 1}$^\circ$] 
The bound in \eqref{eqn:momentbd0} is easy to check using \eqref{eqn:lalley70} with $m=1$ and $x=0$ and the linearity of expectation. In particular, 
\begin{align*}
\E[g_{s}^{N}(0)\mathbf 1_{G_{sN}}]  & \le \frac{sN}{N^{3/2}} \left(1+\sum_{n=1}^{NH} C_{1}n^{1/2-1}\right) = \frac{s}{\sqrt{N}}\left(1+O(\sqrt{N})\right) < \infty.
\end{align*}
For \eqref{eqn:momentbd}, first of all, by triangle inequality and the assumption
  that $\bar X^{N} (x)$ is defined by linear interpolation, we need
  only consider $x, y \in \mathbb Z/\sqrt{N}$ in
  \eqref{eqn:momentbd}.

  Let $Z_{i, n}(x)$ be the number of particles at site
  $x \in \mathbb Z$ in generation $n$ of the $i$-th ancestral
  particle. The left side of \eqref{eqn:momentbd} is clearly bounded by
\begin{align*}
  N^{-3m}\cdot \mathbb E\left \{ \sum_{i=1}^{\lfloor
  sN\rfloor}\sum_{n=0}^{\lfloor NH\rfloor}   \left[ Z_{i,
  n}(\sqrt{N}x)-Z_{i, n}(\sqrt{N}y)\right] \right \}^{2m}. 
\end{align*}
We expand the product under the expectation sign and write it as a sum
of expectations:
\begin{align}
\label{eqn:oneterm}
N^{-3m}\cdot \sum_{i_{1}, \ldots, i_{2m}=1}^{\lfloor sN \rfloor} \sum_{n_{1}, \ldots, n_{2m}=0}^{\lfloor NH \rfloor} \mathbb E\prod_{j=1}^{2m}\left[ Z_{i_{j}, n_{j}}^{\lfloor sN\rfloor}(\sqrt{N}x)-Z_{i_{j}, n_{j}}^{\lfloor sN\rfloor}(\sqrt{N}y)\right].
\end{align}
When the product inside the expectations is expanded, each term is a
product of $2m$ differences of occupation counts in one of the
branching random walks $Z_{i_{j}}$ in some generation $n_{j}$.
Observe that repetitions of the
indices $i_{j}$ and $n_{j}$ are allowed.

Note that Proposition \ref{prop:lalley71} applies only for generation $n\ge 1$, whereas
$n_{1}, \ldots, n_{2m}$ in \eqref{eqn:oneterm} run from generation
$0$. However, since originally all particles are placed at the origin,
we lose nothing by summing from $1$ to $\lfloor NH\rfloor$ as long as
$xy\ne 0$. The case when $xy=0$ will be treated separately at the end.

Suppose $x\ne 0$, $y\ne 0$.  If $i_{j_{1}} \ne i_{j_{2}}$ are indices
of two distinct ancestral individuals, then the differences
$\left[Z_{i_{j_{1}}, n_{j_{1}}}(\sqrt{N}x)-Z_{i_{j_{1}},
    n_{j_{1}}}(\sqrt{N}y)\right]$ and
$\left[Z_{i_{j_{2}}, n_{j_{2}}}(\sqrt{N}x)-Z_{i_{j_{2}},
    n_{j_{2}}}(\sqrt{N}y)\right]$ are independent. Let $r$ be the
number of distinct $i_{j}$'s inside the expectation in \eqref{eqn:oneterm}; then
\eqref{eqn:oneterm} can be written as
\begin{align}
\label{eqn:factors}
N^{-3m}\sum_{r=1}^{2m}\sum_{\substack{1\le i_{1}<\cdots < i_{r}\le \lfloor sN\rfloor\\\sum_{j}m_{j}=2m}}
  \prod_{j=1}^{r}\left[ \sum_{n^{j}_{1}\ldots
  n^{j}_{m_{j}}=1}^{\lfloor NH \rfloor} \mathbb E \prod_{l=1}^{m_{j}}
  \left( Z_{i_{j}, n^{j}_{l}}
  (\sqrt{N}x)-Z_{i_{j}, n^{j}_{l}}
  (\sqrt{N}y)\right)\right]. 
\end{align}
For a particular term with $r$ distinct ancestors
$i_{1}, \ldots, i_{r}$ in which  $i_{j}$ occurs  $m_{j}$
times ($j=1,2,\ldots, r$), the expectation can be factored as a
product of $r$ expectations, where each expectation is an expectation
of the differences involving the offspring of only one ancestor at
time 0. Thus, we always have $\sum_{j=1}^{r}m_{j}=2m$. For each
bracketed factor in  \eqref{eqn:factors}, for each ancestor $i_{j}$, the
summation is over all possible choices of the generations
$n_{1}^{j}, \ldots, n_{m_{j}}^{j}$; this can be bounded using
Corollary  \ref{prop:claim} above. It follows that
\begin{align*}
  &\ \  N^{-3m} \sum_{r=1}^{2m}\sum_{\substack{1\le i_{1}<\cdots < i_{r}\le \lfloor sN\rfloor\\\sum_{j}m_{j}=2m}} \prod_{j=1}^{r}C(m_{j},
    H)|x-y|^{\frac{m_{j}}{5}}N^{\frac{3m_{j}}{2}-1}\\ 
  &\le\  C_{1}(m, s, H) \,N^{-3m} \sum_{r=1}^{2m}N^{r} \cdot
    N^{\frac{3}{2}\cdot
    2m-r}|x-y|^{\frac{2m}{5}}\\  
  &\le \ C(m, s, H)|x-y|^{\frac{2m}{5}},
\end{align*}
and \eqref{eqn:momentbd} is proved. 

Finally, we must deal with the case when $xy=0$. If $x=y=0$, then
both sides of
\eqref{eqn:momentbd} are zero. If $x=0$ and $y\ne 0$, then, because
all initial $\lfloor sN \rfloor$ particles are placed at zero, we can
write the left side of \eqref{eqn:momentbd} as
\[
 \mathbb E\left \{ \left(\sum_{i=1}^{\lfloor sN\rfloor}\sum_{n=1}^{NH}   
\frac{\left[ Z_{i, n}(0)-Z_{i, n}(\sqrt{N}y)\right]}{ N^{3/2}}\right)  +\frac{ \lfloor sN \rfloor}{N^{3/2}}\right \}^{2m}.
\]
It is not difficult to see that for large $N$, the first term
dominates, because this term can be handled exactly as in the case $xy \not =
0$. This proves that Condition 1$^{\circ}$  holds.
\end{proof}

\begin{proof}[Proof of \textup{Condition 2}$^{\circ}$]
  For each $N$ the process $g^{N}_{s}$ is piecewise constant in $s$,
  with jumps only at times $s$ that are integer multiples of
  $1/N$. Consequently, in verifying Condition 2 we may restrict
  attention to stopping times $\tau_{N}$ such that $N \tau_{N}$ is an
  integer between $0$ and $NL$. It is obvious from its definition that
  the discrete-time process $g^{N}_{s}$, with $s=0,1/N, 2/N, \cdots
  $, is non-decreasing and has stationary, independent increments;
  therefore, for any stopping time $\tau^{N}$ and any constant $\delta
  >0$, the increment $g^{N}_{\tau_{N}+\delta}-g^{N}_{\tau_{N}}$ has
  the same distribution as $g^{N}_{\delta}$. Therefore, to prove
  Condition 2$^{\circ}$ it is enough to show that for any $\epsilon
  >0$ there exists $\delta>0$ such that for all $N$ sufficiently large,
  \begin{displaymath}
    \Prob\left (\sup_{x\in \mathbb R}g_{\delta}^{N}(x) \ge \epsilon
    \right) \le \epsilon,
  \end{displaymath}
  equivalently,
  \begin{equation}
    \label{eq:tightness2}
    \Prob\left (\sup_{x\in \mathbb R}g_{1}^{\lfloor \delta N \rfloor}(x) \ge \epsilon \delta^{-3/2}
    \right) \le \epsilon.
  \end{equation}
  But we have already proved, in Condition 1$^{\circ}$, that the
  sequence of $\mathcal{C}_{0}(\mathbb{R})-$valued processes
  $g^{N}_{1}$ is tight, so there exists $K=K_{\epsilon}<\infty$ so
  large that for all $m \in \mathbb{N}$,
  \begin{displaymath}
    \Prob \left (\sup_{x\in \mathbb R}g_{1}^{m}(x) \ge K
    \right) \le \epsilon.
  \end{displaymath}
  By choosing $\delta >0$ so small that $\epsilon/\delta^{3/2}>K$, we
  obtain \eqref{eq:tightness2} for $N$ sufficiently large.
\end{proof}

\subsection{Uniqueness of the Limit Process}Since
$\{g_{s}^{N}\}_{s\ge 0}$  has stationary and
independent increments, any weak limit will also have these
properties. Therefore, to prove the uniqueness of the limit process it
suffices to show that for any fixed time $s>0$ there is only one
possible limit for the sequence $\{g^{N}_{s}\}_{N \in\mathbb{N}}$.

For any $N\in\mathbb{N}$, the random function $g^{N}_{s}(\cdot )$ is
defined by rescaling the occupation measure $X^{\lfloor sN \rfloor}$
of the branching random walk initiated by the first
$\lfloor sN \rfloor$ ancestral individuals (cf. equation
\eqref{eq:occupation-measure-rescaled}). The occupation measure 
$X^{\lfloor sN \rfloor}$ is defined by \eqref{eq:occupation-measure},
which can be rewritten as
\begin{displaymath}
  X^{m}(j)=\sum_{i=1}^{m} \sum_{n=0}^{\infty} Z_{i,n}(j), 
\end{displaymath}
where $Z_{i,n}(j)$ is the occupation counts at location $j$ of individuals in the $n$-th generation of the $i$-th labeled tree. 
As discussed in Remark \ref{rem:indef-integral}, to avoid invoking an indefinite integral operator in the weak limit, we consider the truncated 
occupation counts and the associated occupation density up to the $\lfloor NH\rfloor$-th generation for some $H>0$ fixed. Define
\begin{align*}
  X^{m,H}(j)&=\sum_{i=1}^{m} \sum_{n=0}^{\lfloor NH\rfloor} Z_{i,n}(j) \quad
  \textrm{and} \\
  g^{N}_{s,H}(x)&=N^{-3/2}\bar{X}^{\lfloor sN \rfloor,H} (\sqrt{N}x)
\end{align*}
where, as earlier, the bar denotes the function obtained by linear
interpolation. The same calculations as in the proof of Condition
1$^{\circ }$ show that for any fixed $H>0$ and $s>0$ the sequence
$g^{N}_{s,H}$ is tight. 

Watanabe's convergence theorem states that for any $s>0$, the rescaled measure-valued process converges weakly to the super-Brownian motion, i.e.,
\[
\left\{Y_{t}^{s, N}\right\}_{0\le t\le H} :=\left\{\frac{1}{N}\sum_{j\in \mathbb Z}\sum_{i=1}^{\lfloor Ns\rfloor}Z_{i, \lfloor tN \rfloor}(j)\delta_{j/\sqrt{N}}\right\} \Longrightarrow \left\{Y_{t}^{s}\right\}_{0\le t\le H},
\] 
Viewing the measure-valued process $Y_{t}^{s,N}$ as nonnegative continuous functions over $\mathbb R$, we define $\bar Y_{t}^{s, N}\in \mathcal C_{0}(\mathbb R)$ by setting
\[
\bar Y_{t}^{s, N}(x) := \frac{1}{N}\sum_{i=1}^{\lfloor Ns\rfloor}Z_{i, \lfloor tN \rfloor}(\sqrt{N}x), \quad \text{if } x\in \mathbb Z/\sqrt{N}.
\]
and then doing a linear interpolation. The above convergence implies that the weak convergence of the rescaled total occupation measure of the first $\lfloor NH\rfloor$ generations:
\[
\frac{1}{N}\sum_{t=0}^{\lfloor NH\rfloor} Y^{s, N}_{\frac{t}{N}}(x) \Longrightarrow \int_{0}^{H} Y_{t}^{s}(x) dt. 
\]
Notice that the left side is indeed $\frac{X^{\lfloor sN\rfloor, H}(\sqrt{N}x)}{N^{2}}$, which has
densities $g^{N}_{s,H}$. Consequently, any possible weak
subsequential limit of $\{g^{N}_{s,H}\}_{N\ge 1}$ in the function
space $\mathcal{C}_{0}(\mathbb{R})$ must be a density for the
occupation measure of the super-Brownian motion, that is, as $N\to\infty$
\begin{displaymath}
  g^{N}_{s,H} \Longrightarrow g_{s,H} \quad \textrm {where} \quad g_{s,
    H}(x)\mathdef \int_{0}^{H}Y^{s}(t, x)dt .  
\end{displaymath}
 But by
inequality \eqref{eq:kolmogorov-tail}, for any $\epsilon>0$ there
exists $H=H_{\epsilon}<\infty$ so large that for any $N$,
\begin{displaymath}
  \Prob \xset {g^{N}_{s}\not = g^{N}_{s,H}}\leq \epsilon.
\end{displaymath}
Consequently, the sequence $\{g^{N}_{s}\}_{N\in \mathbb N}$ must converge weakly to
$\int_{0}^{\infty}Y^{s}(t, x)dt$.

\section{\label{sec:Levyjump}Properties of the Limiting Process}
In this section, we prove properties of the limiting process $\{g_{s}\}_{s\ge 0}$ (Theorem \ref{thm:property-gs}) and characterize it using a Poisson point process (Theorem \ref{thm:levy-measure}). In order to make sense of the notion of
a ``subordinator'' on the function space $\mathcal C_{0} (\zz{R})$, we first briefly
review the definition of a \emph{Banach lattice}.

\begin{definition}\label{def:banachlattice} A {Banach lattice} is a
triple $(E, \lVert \cdot\rVert, \le)$ such that
\begin{enumerate}[\normalfont(a).]
\item  $(E, \lVert \cdot \rVert)$ is a Banach space with norm $\lVert\cdot\rVert$;
\item $(E, \le)$ is an ordered vector space with the partial ordering $\le$;
\item under $\le$, any pair $x, y \in E$ has a least upper bound
denoted by $x \vee y$ and a greatest lower bound denoted by $x \wedge
y$ (this is the ``lattice'' property); and 
\item  Set  $|x| \mathdef x \vee (-x)$. Then $|x|\le |y|$ implies
$\lVert x \rVert \le \lVert y\rVert$, $\forall x, y \in E$ (i.e.,
$\lVert \cdot \rVert$ is ``a
lattice norm''). 
\end{enumerate}
\end{definition}

\begin{example}
The Banach space $(\mathcal C_{0}(\mathbb R), \lVert \cdot
\rVert_{\infty})$ has a natural  partial ordering, defined by
\[
f \le g \textup{ if and only if } g(x)-f(x) \ge 0 \quad \text{for all}
\quad x\in \mathbb R.
\]
The triple $(\mathcal C_{0}(\mathbb R), \lVert \cdot \rVert_{\infty},
\le)$ clearly satisfies (a) and (b) in Definition
\ref{def:banachlattice}. The least upper bound and the greatest lower
bound are defined pointwise:
\[
(f\vee g) (x) = f(x) \vee g(x), \qquad (f\wedge g) (x) = f(x) \wedge g(x).
\]
Condition (d) can be verified easily. 
\end{example}

\begin{definition} Let $(E, \lVert \cdot\rVert, \le)$ be a Banach
lattice. An $E$-valued stochastic process $\{X_{t}\}_{t\ge 0}$ is a
\emph{subordinator} if $\{X_{t}\}_{t \ge 0}$ is a L\'evy process (that
is, $\{X_{t}\}_{t \ge 0}$ has stationary, independent increments) and with probability one,
for all $t \ge s \ge 0$,
\[
X_{t}-X_{s} \ge 0.
\]
A subordinator $\{X_{t}\}_{t\ge 0}$ is a \emph{pure jump} process if for every $t$,
\[
X_{t} = \sum_{s\le t}(X_{s}-X_{s-}).
\] 
\end{definition}

\begin{proof}[Proof of Theorem \ref{thm:property-gs}]
For (i), we have for each $N\ge 1$, 
\[
g_{s}^{N}(x) = \frac{\bar X^{\lfloor sN \rfloor}(\sqrt N x)}{N^{3/2}} = \frac{\lfloor sN\rfloor ^{3/2}}{N^{3/2}}\cdot \frac{\bar X^{1\cdot \lfloor sN \rfloor}\left (\sqrt{N} x\right)}{\lfloor sN\rfloor ^{3/2}} = \frac{\lfloor sN\rfloor ^{3/2}}{N^{3/2}} g_{1}^{\lfloor sN\rfloor}\left (\sqrt{\frac{N}{\lfloor sN \rfloor}}\, x \right).
\]
Taking $N\to\infty$ gives (i). The claim that $\{I_{s}\}_{s\ge 0}$ is a stable-2/3 subordinator follows from monotonicity of $g_{s}$ and the scaling relation above at $x=0$, which yields
\[
I_{s} = s^{3/2} I_{1}.
\]
For (iii), recall that a version of the stable--$\frac{1}{2}$ subordinator on $\zz{R}$ is the inverse local-time process of a standard Brownian motion $\{B_{t}\}_{t\ge 0}$
\[
	\tilde{\tau}_{s}=\inf \{t\geq 0\,:\, L^{0}_{t}>s\},
\]
where $L^{0}_{t}$ is the Brownian local time at location $0$ up to time $t$. 
The jumps of the process $\{\tilde{\tau}_{s} \}_{s\geq 0}$ are the
lengths of the excursions of the Brownian path.

Now consider a sequence of independent critical Galton-Watson trees
$\mathcal{T}_{i}$ with offspring distribution $\nu$, initiated by
particles $i=1,2,3,\dots$. Let $|\mathcal{T}_{i}|$ be the size (number of
vertices) of the $i$-th tree, and set $A_{N}\mathdef\sum_{i=1}^{N}|\mathcal{T}_{i}|$, 
the total number of vertices in the first $k$ trees. 
Then by a theorem of Le Gall \cite{LeGall2005}, as $N\to\infty$,
\[
	\{ A_{\lfloor sN\rfloor}/N^{2}\}_{s\geq 0} \Longrightarrow \{\tilde{\tau}_{s/\sigma_{\nu}}
	\}_{s\geq 0} \equalD \{\sigma_{\nu}^{-2}\tilde{\tau}_{s}
	\}_{s\geq 0},
\]
where the last equality follows from the scaling rule of a stable--$\frac{1}{2}$ process.

Next, suppose that branching random walks are built on the
Galton-Watson trees $\mathcal{T}_{i}$ by labelling the vertices,
as described earlier.  
Then clearly
\[
A_{\lfloor sN\rfloor } = \sum_{j\in \zz{Z}/\sqrt{N}}X^{\lfloor sN\rfloor} (\sqrt{N}j).
\]
By Theorem \ref{thm1.1}, $\{g_{s}^{N}\}_{s\ge 0} \Rightarrow \{g_{s}\}_{s\ge 0}$ in $\mathbb D([0, +\infty), \mathcal C_{0}(\mathbb R))$. Considering the space-truncated occupation density $g_{s}^{N}(x)\mathbf 1_{[-B, B]}(x)$ (i.e., truncated in space) for sufficiently large $B>0$ and following the same strategy as when proving the uniqueness of the limiting process $\{g_{s}\}_{s\ge 0}$, one would obtain 
\[
 \left\{\frac{1}{\sqrt{N}}\sum_{x\in \zz{Z}/\sqrt{N}}g_{s}^{N}(x)\right\}_{s\geq 0} \Longrightarrow \left\{ \int_{\zz R} g_{s}(x)dx\right\}_{s\ge 0} := \{\theta_s\}_{s\ge 0},
\]
where the left side is indeed 
\[
\frac{1}{\sqrt{N}}\sum_{x\in \zz{Z}/\sqrt{N}}g_{s}^{N}(x) = \frac{\sum_{x\in \zz{Z}/\sqrt{N}}X^{\lfloor sN\rfloor}(\sqrt{N}x)}{N^{2}} = \frac{A_{\lfloor sN\rfloor}}{N^{2}}. 
\]

Consequently, the processes $\{\theta_{s}\}_{s\ge 0}$ and
$\{\sigma_{\nu}^{-2}\tilde{\tau}_{s} \}_{s\geq 0}$ have the same law, and so
$\{\theta_{s}\}_{s\ge 0}$ is a stable-1/2 subordinator.

For (iv), we have already observed that $g_{s}$ has stationary, independent
increments and increasing sample paths relative to the natural partial
order $\leq$ on $\mathcal C_{0} (\zz{R})$. 
To show that $\{g_{s}\}_{s\ge 0}$ has pure jumps, we make use of the fact that the 
total area process $\{\theta_{s}\}_{s\ge 0}$ is a stable--$\frac{1}{2}$ subordinator 
and thus has pure jumps. Let $\mathcal J$ be the set of jump times of the process $\{\theta_{s}\}_{s\geq 0}$, that is, the set of all $t\geq 0$ for which $\theta
(t)-\theta (t-)>0$. Define
\[
	\tilde{g}_{s}=\sum_{t\in\mathcal  J\cap [0,s]} (g_{t}-g_{t-}),
\]
a process that collects the changes in $\{g_{s}\}_{s\ge 0}$ at those times when the limiting total area process $\{\theta_{s}\}_{s\ge 0}$ makes jumps.
Clearly, the process $\tilde{g}_{s}$ is an increasing process in
$\mathcal C_{0} (\zz{R})$, and since $\tilde g_{s}$ only gathers the 
jumps of $g_{s}$, we have
\[
	\tilde{g}_{s}\leq g_{s} \quad \text{for every} \;\; s\geq 0.
\]
But since the area process $\theta_{s}$ is pure jump, 
$g_{s}$ and $\tilde{g}_{s}$ bound the same total area for every $s$,
that is,
\[
	\int_{\zz{R}} g_{s} (x)\,dx = \theta_{s}= \int_{\zz{R}} \tilde{g}_{s} (x)\,dx.
\]
By continuity of both $g_{s}$ and $\tilde g_{s}$, we have $g_{s}=\tilde{g}_{s}$ for every $s$, and thus the process
$g_{s}$ is a pure jump process in $\mathcal C_{0} (\zz{R})$.

\end{proof}

\begin{proof}
[Proof of Theorem~\ref{thm:levy-measure}] We have already proved in
Theorem~\ref{thm:property-gs} that the process $g_{s}$ consists of pure
jumps. It remains to show that the point
process of jumps is a Poisson point process with intensity given 
by \eqref{eqn:gs-intensity} and then the representation 
\eqref{eqn:pointprocess} would follow automatically.

Consider the point process of jumps of $g_{s}$ for $s\leq 1$
(the case $s\leq s_{*}$, for arbitrary $s_{*}>0$, can be handled in
analogous fashion). Let $J_{1}, J_{2},\dots$ be the jumps ordered by
size from largest to smallest, as in Corollary~\ref{cor:order-stats}. Since by
Theorem~\ref{thm:property-gs},  the limiting process  $\{g_{s}\}_{s\ge 0}$ is a pure
jump subordinator, we have 
\begin{align}
\label{eqn:repg1}
	g_{1}=\sum_{i=1}^{\infty}J_{i}.
\end{align}
Theorem~\ref{thm:property-gs} also implies that the jump sizes
$|J_{i}|=\int J_{i} (x)\,dx$ are distributed (jointly) as the ordered excursion lengths
of a standard Brownian motion run up to the first time $t$ that
$L^{0}_{t}=1$, rescaled by $\sigma_{\nu}^{-2}$. By Corollary~\ref{cor:order-stats}, for any $m\geq 1$,
as $k \rightarrow \infty$,
\begin{align}
\label{eqn:jump-cov}
	(J^{N}_{1},J^{N}_{2},\dots ,J^{N}_{m})\Longrightarrow 
	(J_{1},J_{2},\dots ,J_{m}),
\end{align}
where $J^{N}_{1},J^{N}_{2},\dots$ are the ordered jumps in the (rescaled)
occupation density processes $g^{N}_{s}$ for $s\leq 1$ for the
branching random walk obtained by amalgamating the first $N$
trees.
Consequently, the joint distribution of the random variables
\[
	|J^{N}_{i}|=|\mathcal{T}^{N}_{(i)}|/N^{2}
\]
(where $\mathcal{T}^{N}_{(i)}$ is the $i$-th largest tree among the first
$N$ trees) converges to the joint distribution of the sizes $|J_{i}|$.
In particular, the largest, second largest, etc., trees among the
first $N$ trees have sizes of order $N^{2}$ --- and so as $N
\rightarrow \infty$, these will be large. 

To identify the limiting distribution of the rescaled jumps, we now
make use of Theorem 1.1 in \cite{condGW-profile}, which states
 that the occupation density of a conditioned branching random
walk scaled by the size of the tree converges to that of the 
ISE density $f_{\textup{ISE}}$, as the size of the tree becomes large.
This implies, for each $i=1,2,\dots $, as $N \rightarrow \infty$,
\begin{align}
\label{eqn:jump-char}
	\frac{J^{N}_{i} (|J_{i}^{N}|^{1/4}\ \cdot\, )}{|J_{i}^{N}|^{3/4}}
	\Longrightarrow \gamma f^{(i)}_{\textup{ISE}}(\gamma \ \cdot\,),\ \ 
	\text{where } \gamma = \sigma_{F}^{-1}\sigma_{\nu}^{1/2},
\end{align}
and the limiting ISE densities, $ f^{(1)}_{\textup{ISE}}$, $ f^{(2)}_{\textup{ISE}}$, $\ldots$ are i.i.d copies of $f_{\textup{ISE}}$. By \eqref{eqn:jump-cov} and \eqref{eqn:jump-char}, we can describe the joint distribution of $J_{1},J_{2},\dots$ as follows: (a) let
$\varepsilon_{1}>\varepsilon_{2}>\varepsilon_{3}>\dots $ be the
ordered excursion lengths of a standard Brownian motion run until the
first time $t$ such that $L^{0}_{t}=1$;  (b) let
$f_{1},f_{2},f_{3},\dots$ be i.i.d.
copies of the ISE density $f_{\textup{ISE}}$ which are independent of the $\varepsilon_{i}$'s; and (c) set 
\begin{align*}
	J_{i} (\cdot) &\mathdef (\sigma_{\nu}^{-2}\varepsilon_{i})^{3/4}\gamma f_{i} (\gamma (\sigma_{\nu}^{-2}\varepsilon_{i})^{-1/4}\ \cdot\,)\\
	&= \frac{1}{\sigma_{\nu}\sigma_{F}}\, \varepsilon_{i}^{3/4} f_{i} (\sigma_{\nu}\sigma_{F}^{-1}\varepsilon_{i}^{-1/4}\ \cdot\,).
\end{align*}
Since the ordered excursion lengths
$\varepsilon_{1}>\varepsilon_{2}>\varepsilon_{3}>\dots $  have the
distribution of the ordered points in a Poisson point process on
$\zz R_{+}$ with intensity measure
$
	\frac{dy}{\sqrt{2\pi y^{3}}},
$
the representation \eqref{eqn:pointprocess} follows from \eqref{eqn:repg1}.

\end{proof}

\noindent{\bf Acknowledgment.} The authors are grateful to the anonymous referee for valuable comments. }






\end{document}